\documentclass[11pt]{article}
\usepackage[T1]{fontenc}
\usepackage[utf8]{inputenc}
\usepackage{lmodern}
\usepackage[a4paper,margin=1in]{geometry}
\usepackage{amsmath,amssymb,amsthm,mathtools}
\usepackage{microtype,graphicx,booktabs,enumitem}
\usepackage[numbers,sort&compress]{natbib}
\usepackage{xcolor}
\usepackage[colorlinks=true,linkcolor=blue!50!black,citecolor=blue!50!black,urlcolor=blue!50!black]{hyperref}
\usepackage[nameinlink,noabbrev]{cleveref}
\numberwithin{equation}{section}
\newtheorem{theorem}{Theorem}[section]
\newtheorem{lemma}[theorem]{Lemma}
\newtheorem{proposition}[theorem]{Proposition}
\newtheorem{corollary}[theorem]{Corollary}
\theoremstyle{remark}
\newtheorem{remark}[theorem]{Remark}
\newcommand{\R}{\mathbb R}

\newcommand{\B}{B_2^n}
\newcommand{\E}{\mathcal E}
\DeclareMathOperator{\vol}{vol}
\DeclareMathOperator{\tr}{tr}
\DeclareMathOperator{\diag}{diag}
\DeclareMathOperator{\interior}{int}
\newcommand{\rstar}{r_{\!*}}
\setlist{itemsep=3pt,topsep=5pt}
\newcommand{\paperauthors}{%
Zhou Longfei\thanks{\href{mailto:longfei.gabriel.zhou@gmail.com}{\texttt{longfei.gabriel.zhou@gmail.com}}}%
\and Haijun Zou\thanks{Academy of Mathematics and Systems Science, Chinese Academy of Sciences, Beijing, China; University of Chinese Academy of Sciences, Beijing, China; \href{mailto:zouhaijun24@mails.ucas.ac.cn}{\texttt{zouhaijun24@mails.ucas.ac.cn}}}%
\and Tianhao Liu\thanks{Antai College of Economics and Management, Shanghai Jiao Tong University; \href{mailto:tianhao.liu@sjtu.edu.cn}{\texttt{tianhao.liu@sjtu.edu.cn}}}}

\title{A Spectral Proof of Khachiyan's Ellipsoid Conjecture}
\author{\paperauthors}
\date{September 23, 2026}
\hypersetup{pdftitle={A Spectral Proof of Khachiyan's Ellipsoid Conjecture},pdfauthor={Zhou Longfei, Haijun Zou, Tianhao Liu},pdfsubject={Sharp ellipsoid cutting bound with Lean 4 formal verification}}

\begin{document}
\maketitle
\begin{abstract}
For a convex body $K\subset\R^n$, let $w(K)$ denote the volume of its maximum-volume inscribed ellipsoid. We prove that every closed halfspace $H$ whose boundary passes through the center of the maximizing ellipsoid satisfies
\[
 w(K\cap H)\le\frac{\sqrt e}{2}\,w(K).
\]
The constant is optimal uniformly over all dimensions, as witnessed by a family of circular cones, thereby establishing Khachiyan's conjecture.

\medskip\noindent The proof converts containment, maximality, and the central-cut condition into algebraic constraints on positive definite matrices. Two complementary spectral bounds from a diagonal model extend to arbitrary center displacements through a directional rank-one estimate for fractional trace powers. Concavity determines their joint optimum. We also derive finite-dimensional bounds and a necessary condition for near equality, with self-contained supporting proofs and an alternative resolvent argument. An AI language model discovered the proof in a human-directed research process. Lean~4 with mathlib verifies the main theorem, sharpness, and supporting results (\href{https://github.com/DrZhouKarl/KhachiyanEllipsoidConjecture}{GitHub source}).
\end{abstract}

\noindent\textbf{Keywords:} maximum-volume inscribed ellipsoid; Khachiyan's conjecture; convex body; cutting plane; fractional trace power.

\section{Main result and proof strategy}\label{sec:introduction}

A hyperplane through the center of a body's maximum-volume inscribed ellipsoid determines two closed half-bodies. We prove that the maximum-volume ellipsoid inscribed in either half-body has volume at most $\rstar=\sqrt e/2=0.8243606353\ldots$ times that of the original maximizing ellipsoid, and that this dimension-independent constant is optimal.

To state the result precisely, let $K\subset\R^n$ be compact and convex with nonempty interior, and write
\begin{equation}\label{eq:w}
 w(K)=\max\{\vol(E):E\subset K,\ E\text{ is an ellipsoid}\}.
\end{equation}
The maximizing ellipsoid is unique. Denote it by $J(K)$ and its center by $c(K)$; we use the customary name \emph{John ellipsoid}. Existence and uniqueness are proved in \cref{app:ellipsoids}.

\begin{theorem}\label{thm:main}
Let $n\ge1$, let $K\subset\R^n$ be a convex body, and let $H$ be a closed halfspace whose boundary contains $c(K)$. Then
\begin{equation}\label{eq:main}
 w(K\cap H)\le\rstar\,w(K),\qquad \rstar=\frac{\sqrt e}{2}.
\end{equation}
Moreover, $\rstar$ is the smallest constant for which this inequality holds simultaneously for all dimensions, convex bodies, and such halfspaces.
\end{theorem}

Let $H^+$ and $H^-$ be the two closed halfspaces determined by the cutting hyperplane. The theorem bounds $w(K\cap H^+)$ and $w(K\cap H^-)$ separately. The functional $w$ is not generally additive under this decomposition: the identity $w(K\cap H^+)+w(K\cap H^-)=w(K)$ does not hold for arbitrary convex bodies. The proof must account for both center displacement and changes in the semiaxes of the new maximizing ellipsoids.

The constant was conjectured by \citet{Khachiyan1990}, who proved a bound of $0.844\ldots$. The earlier inscribed-ellipsoid method of \citet{TarasovKhachiyanErlikh1988} introduced an intermediate-ellipsoid construction, its recursive application, and the cone example with limiting ratio $\sqrt e/2$. The sharp estimate below follows from retaining \emph{two complementary constraints} supplied by that construction. A diagonal model makes their relationship explicit.

After normalizing $J(K)$ to the unit ball, a candidate ellipsoid has the form $a+A\B$ and its relative volume is $\det A$. If its center lies along a shortest principal axis, the two intermediate ellipsoids admit explicit diagonal representations. Each yields an upper bound on $\det A$ in terms of the same smallest semiaxis $\alpha$. Neither bound alone gives the optimal constant, but the maximum of their pointwise minimum does. For an arbitrary displacement, the intermediate matrices need not be simultaneously diagonalizable. A rank-one matrix inequality nevertheless establishes the same trace estimates, which suffice for the determinant argument.

The argument thus converts geometric feasibility into algebraic constraints: containment supplies feasible shape matrices, maximality bounds their traces, and a common spectral parameter relates the two estimates. Concavity certifies the final optimization, while the cone family establishes optimality. A Lean~4 formalization verifies the main theorem and its supporting results; \cref{app:formal-verification} records the scope and reproducibility of the verification.

\paragraph{Organization of the proof.}
\Cref{sec:architecture} derives the two scalar bounds in the diagonal model. \Cref{sec:upper} establishes the trace comparison needed for arbitrary displacements, and \cref{sec:sharpness} proves sharpness. Together, these sections present the complete argument for the main result. Supporting lemmas are stated at their first use, with proofs collected in \cref{sec:estimates}. Further consequences appear in \cref{sec:consequences}, and related work and open questions in \cref{sec:history}. The appendices supply the standard matrix and ellipsoid facts and an alternative proof of the rank-one estimate.

\section{Two spectral bounds from a diagonal model}\label{sec:architecture}

\subsection{Affine normalization and the determinant ratio}
An invertible affine change of coordinates takes $J(K)$ to the Euclidean unit ball $\B$ and its center to $0$. It preserves containment and multiplies every volume by the same factor. We may therefore assume $J(K)=\B$ and write a candidate ellipsoid as
\[
 \E(a,A)=a+A\B,\qquad A=A^\top\succ0.
\]
Every nonsingular ellipsoid admits this representation: $C\B=(CC^\top)^{1/2}\B$ for any nonsingular $C$. The eigenvalues of $A$ are its semiaxis lengths, so
\[
 \frac{\vol(\E(a,A))}{w(K)}=\det A.
\]
We use the Loewner order for symmetric matrices and spectral calculus for their powers. Thus $\tr(A^p)$ means the trace after taking the matrix power.

The central-cut condition has an equivalent algebraic formulation. An ellipsoid contained in a closed halfspace whose boundary passes through $0$ cannot contain $0$ in its interior. In ellipsoid coordinates, this condition is $\|A^{-1}a\|\ge1$.

\begin{proposition}[Center-exclusion formulation]\label{prop:exclusion}
The upper bound in \cref{thm:main} is equivalent to the implication
\begin{equation}\label{eq:excluded-center}
 \E(a,A)\subset K,\quad J(K)=\B,\quad \|A^{-1}a\|\ge1
 \quad\Longrightarrow\quad\det A\le\rstar.
\end{equation}
\end{proposition}

The equivalence is proved in \cref{sec:basic-proofs}. It expresses the geometric restriction through a displacement norm, eliminating the cutting normal from the subsequent estimates.

\subsection{Two geometric ingredients}
The construction uses two consequences of geometric feasibility. First, maximality of the unit ball bounds the trace of every feasible shape matrix.

\begin{lemma}[Trace constraint]\label{lem:trace}
If $J(K)=\B$ and $\E(b,S)\subset K$ with $S\succ0$, then $\tr S\le n$.
\end{lemma}

The second constructs a new feasible ellipsoid from the unit ball and any other ellipsoid inside $K$.

\begin{lemma}[Intermediate ellipsoid]\label{lem:intermediate}
If $\B\subset K$ and $\E(b,S)\subset K$, where $S\succ0$, then
\begin{equation}\label{eq:intermediate}
 \E\!\left(\frac b2,\left(S+\frac14bb^\top\right)^{1/2}\right)\subset K.
\end{equation}
\end{lemma}

Proofs appear in \cref{sec:basic-proofs,sec:intermediate-proof}. The construction is the normalized form of that in \citet{TarasovKhachiyanErlikh1988,Khachiyan1990}. The rank-one term $bb^\top/4$ encodes the center displacement in the new shape matrix, to which the trace constraint can then be applied.

Apply the construction first to $\E(a,A)$, and then to the resulting ellipsoid, always pairing it with the \emph{same unit ball}. The new centers are $a/2$ and $a/4$, and their shape matrices are
\begin{equation}\label{eq:T-definition}
 T_1=(A+aa^\top/4)^{1/2},\qquad
 T_2=(T_1+aa^\top/16)^{1/2}.
\end{equation}
The coefficient $1/16$ in the second expression results from applying the construction to the center $a/2$. Since both ellipsoids remain in $K$, \cref{lem:trace} gives two necessary conditions:
\begin{equation}\label{eq:T-traces}
 \tr T_1\le n,\qquad\tr T_2\le n.
\end{equation}

The auxiliary ellipsoids may intersect both sides of the cutting hyperplane. Their role is to provide consequences of maximality of the original ball, so only containment in $K$ is required. \Cref{fig:geometry} shows this distinction in dimension two.

\begin{figure}[tb]
 \centering
 \includegraphics[width=\textwidth]{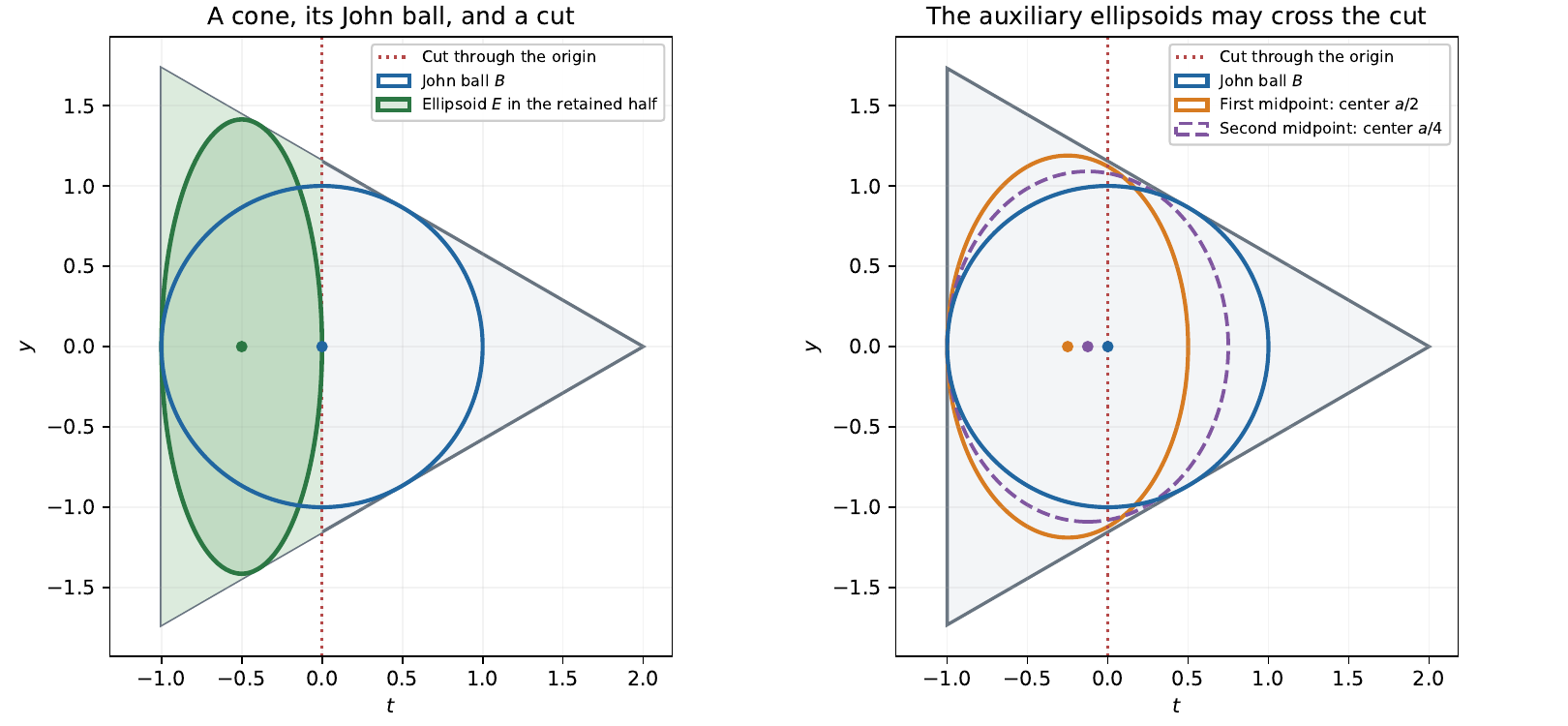}
 \caption{The construction for the two-dimensional cone used in \cref{sec:sharpness}. Left: the unit ball and an ellipsoid in the retained half. Right: the two intermediate ellipsoids remain in the cone even though they cross the cut. Their purpose is to supply the trace constraints~\eqref{eq:T-traces}.}
 \label{fig:geometry}
\end{figure}

\subsection{The aligned model}
First consider a feasible candidate satisfying
\begin{equation}\label{eq:diagonal-model}
 A=\diag(\alpha,\lambda_2,\ldots,\lambda_n),\qquad
 a=\alpha e_1,\qquad \alpha=\lambda_{\min}(A).
\end{equation}
The normalized displacement has length one: $\|A^{-1}a\|=1$. Both intermediate matrices are diagonal, and the rank-one perturbation affects only their first diagonal entry. The remaining entries are obtained by successive square roots. Consequently,
\begin{align}
 T_1&=\diag\bigl(f_1(\alpha),\lambda_2^{1/2},\ldots,\lambda_n^{1/2}\bigr),\notag\\
 T_2&=\diag\bigl(f_2(\alpha),\lambda_2^{1/4},\ldots,\lambda_n^{1/4}\bigr),\label{eq:diagonal-iterates}
\end{align}
where
\begin{equation}\label{eq:f-definition}
 f_1(\alpha)=\sqrt{\alpha+\alpha^2/4},\qquad
 f_2(\alpha)=\sqrt{f_1(\alpha)+\alpha^2/16}.
\end{equation}

Substituting these diagonal representations into the trace constraints gives
\begin{equation}\label{eq:architecture-spectrum}
 f_j(\alpha)+\sum_{i=2}^n\lambda_i^{1/2^j}\le n,
 \qquad j=1,2.
\end{equation}
These constraints have so far been established for the aligned model. \Cref{sec:upper} proves them for arbitrary center displacements. The following determinant calculation depends only on these spectral sums and therefore applies in either setting.

\subsection{From trace constraints to volume bounds}
The volume ratio is a product of relative semiaxes, whereas the trace constraints bound sums of their fractional powers. Taking logarithms and using $\log z\le z-1$ yields
\begin{align}
 \log\det A
 &=\log\alpha+2^j\sum_{i=2}^n\log(\lambda_i^{1/2^j})\notag\\
 &\le\log\alpha+2^j\left(\sum_{i=2}^n\lambda_i^{1/2^j}-(n-1)\right)\notag\\
 &\le\log\alpha+2^j(1-f_j(\alpha)).\label{eq:logdet-bound}
\end{align}
The contribution of the $n-1$ remaining semiaxes cancels the dimensional term in the trace bound. The resulting estimate depends only on $\alpha$:
\begin{equation}\label{eq:R-definition}
 \det A\le\min\{R_1(\alpha),R_2(\alpha)\},\qquad
 R_j(\alpha)=\alpha\exp\!\bigl(2^j(1-f_j(\alpha))\bigr).
\end{equation}

Both bounds constrain the same ellipsoid and therefore the same value of $\alpha$. The maximum of $R_1$ alone is $0.8442948676\ldots$, the constant in Khachiyan's 1990 argument. The maximum of $R_2$ alone is approximately $0.8248795633$. Both exceed the sharp constant. The improvement follows from bounding their pointwise minimum, as illustrated in \cref{fig:bounds}, rather than optimizing the two bounds separately.

At $\alpha=1/2$, the distinguished semiaxes and the two bounds have the exact values
\[
 f_1(1/2)=\frac34,\qquad f_2(1/2)=\frac78,
 \qquad R_1(1/2)=R_2(1/2)=\frac{\sqrt e}{2}.
\]
The following lemma gives the required \emph{global} conclusion. Its proof in \cref{sec:scalar-proof} uses concavity in the variable $\log\alpha$; the plot is only an illustration.

\begin{lemma}[Joint scalar bound]\label{lem:scalar}
For the functions defined in~\eqref{eq:R-definition},
\begin{equation}\label{eq:scalar-sharp}
 \max_{\alpha>0}\min\{R_1(\alpha),R_2(\alpha)\}
 =\frac{\sqrt e}{2},
\end{equation}
with the maximum attained at $\alpha=1/2$.
\end{lemma}

\begin{figure}[tb]
 \centering
 \includegraphics[width=.94\textwidth]{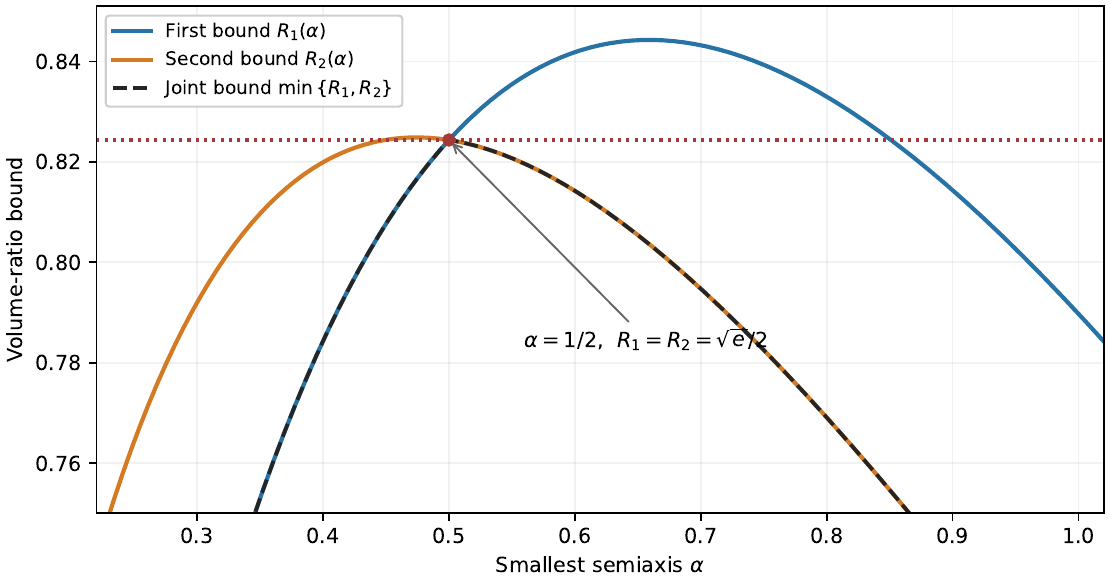}
 \caption{The pointwise minimum of $R_1$ and $R_2$ attains its maximum at $\alpha=1/2$, although neither function individually attains its maximum there. The global bound is proved in \cref{sec:scalar-proof}.}
 \label{fig:bounds}
\end{figure}

This establishes the upper-bound argument for the diagonal model, subject to the supporting lemmas proved in \cref{sec:estimates}. Extending the argument requires only the spectral constraints~\eqref{eq:architecture-spectrum} for arbitrary center displacements; the subsequent determinant calculation is unchanged.

\section{Extension to arbitrary center displacements}\label{sec:upper}

The determinant calculation in the preceding section depends only on the eigenvalues of $A$. For a general displacement, however, diagonalizing $A$ need not diagonalize $aa^\top$. The explicit eigenvalue expressions for $T_1$ and $T_2$ in the aligned model therefore need not apply.

The required spectral constraints concern only the two sums in~\eqref{eq:architecture-spectrum}. They can be established by trace comparisons, without determining the individual eigenvalues of either intermediate matrix.

\subsection{Displacement normalization and matrix comparison}
Fix an arbitrary feasible candidate satisfying the hypotheses in~\eqref{eq:excluded-center}. Put
\begin{equation}\label{eq:normalized-displacement}
 \rho=\|A^{-1}a\|\ge1,\qquad
 u=\rho^{-1}A^{-1}a,\qquad b=Au=\rho^{-1}a.
\end{equation}
Then $\|u\|=1$ and $bb^\top\preceq aa^\top$. Shortening the displacement in this way gives comparison matrices
\begin{equation}\label{eq:M-iterates}
 M_1=(A+bb^\top/4)^{1/2},\qquad
 M_2=(M_1+bb^\top/16)^{1/2}.
\end{equation}
Square-root monotonicity gives $M_1\preceq T_1$ and then $M_2\preceq T_2$. Therefore the trace bounds for the actual intermediate ellipsoids imply
\begin{equation}\label{eq:M-traces}
 \tr M_1\le n,\qquad\tr M_2\le n.
\end{equation}
The reduction uses matrix order alone and makes no feasibility claim for an ellipsoid obtained by replacing $a$ with $b$. Its purpose is to express the displacement as $b=Au$ with $u$ a unit vector, as required by the next lemma.

\subsection{A directional trace estimate}
For $0<p<1$ and fixed $h>0$, the increment $(z+h)^p-z^p$ decreases with $z>0$. The following matrix estimate retains quantitative dependence on both the baseline matrix and the perturbation direction.

\begin{lemma}[Directional rank-one estimate]\label{lem:directional}
Let $A\succ0$, $\alpha=\lambda_{\min}(A)$, $\|u\|=1$, and $b=Au$. If $M\succ0$ and $m>0$ satisfy
\begin{equation}\label{eq:M-comparison}
 M\preceq\frac m{\alpha^2}A^2,
\end{equation}
then, for $0<p<1$ and $t\ge0$,
\begin{equation}\label{eq:directional}
 \tr\bigl((M+tbb^\top)^p\bigr)-\tr(M^p)
 \ge(m+t\alpha^2)^p-m^p.
\end{equation}
\end{lemma}

The proof is in \cref{sec:rank-one-proof}, with another proof in \cref{app:resolvent}. Its mechanism is the operator antitonicity of the derivatives of fractional power functions, which also underlies the trace-increment inequalities in \citet{BouhtouGaubertSagnol2010,Sagnol2013}. The additional comparison with $A^2$ gives the explicit scalar lower bound needed here. The lemma applies without requiring $A$, $M$, and $bb^\top$ to commute.

Three applications suffice. The first bounds the trace of $M_1$; the other two combine through cancellation of intermediate terms to bound the trace of $M_2$.

\begin{proposition}[Trace bounds for arbitrary displacements]\label{prop:spectral}
For the comparison matrices in~\eqref{eq:M-iterates} and $\alpha=\lambda_{\min}(A)$,
\begin{align}
 \tr M_1&\ge\tr(A^{1/2})+f_1(\alpha)-\alpha^{1/2},\label{eq:trace-first}\\
 \tr M_2&\ge\tr(A^{1/4})+f_2(\alpha)-\alpha^{1/4}.\label{eq:trace-second}
\end{align}
\end{proposition}

\begin{proof}
\emph{First iterate.} Since $A\succeq\alpha I$, spectral calculus gives $A\preceq A^2/\alpha$. In \cref{lem:directional}, take
\[
 (M,m,p,t)=(A,\alpha,1/2,1/4).
\]
The left-hand side is $\tr M_1-\tr(A^{1/2})$, and the scalar term on the right is $f_1(\alpha)-\alpha^{1/2}$. This is the first trace bound~\eqref{eq:trace-first}.

\emph{Second iterate.} To apply the lemma with baseline $M_1$, observe that $bb^\top\preceq A^2$ implies
\begin{equation}\label{eq:M1-upper}
 M_1\preceq\left(\frac1\alpha A^2+\frac14A^2\right)^{1/2}
 =\frac{f_1(\alpha)}\alpha A
 \preceq\frac{f_1(\alpha)}{\alpha^2}A^2.
\end{equation}
We can therefore apply \cref{lem:directional} both with $(A,\alpha,1/4,1/4)$ and with $(M_1,f_1(\alpha),1/2,1/16)$. The two conclusions are
\begin{align*}
 \tr(M_1^{1/2})-\tr(A^{1/4})
 &\ge f_1(\alpha)^{1/2}-\alpha^{1/4},\\
 \tr M_2-\tr(M_1^{1/2})
 &\ge f_2(\alpha)-f_1(\alpha)^{1/2}.
\end{align*}
Adding them cancels $\tr(M_1^{1/2})$ on the left and $f_1(\alpha)^{1/2}$ on the right. The remaining inequality is precisely the second trace bound~\eqref{eq:trace-second}.
\end{proof}

The cancellation relates the two intermediate constructions. A fourth-root comparison with the original matrix and a square-root comparison with the first iterate yield exactly the scalar expression obtained in the diagonal model.

\subsection{Completion of the upper bound}
Combine \cref{prop:spectral} with the upper trace bounds~\eqref{eq:M-traces}. If $\lambda_1=\alpha,\lambda_2,\ldots,\lambda_n$ are the eigenvalues of $A$, the result is
\[
 f_j(\alpha)+\sum_{i=2}^n\lambda_i^{1/2^j}\le n,
 \qquad j=1,2.
\]
These are precisely the constraints in~\eqref{eq:architecture-spectrum}, established for arbitrary displacements. The determinant calculation~\eqref{eq:logdet-bound} thus applies unchanged and gives
\[
 \det A\le\min\{R_1(\alpha),R_2(\alpha)\}
 \le\frac{\sqrt e}{2}
\]
by \cref{lem:scalar}. The convention of an empty sum covers $n=1$. The argument also permits a repeated smallest eigenvalue, since only one occurrence of $\alpha$ is separated from the sum.

This proves the upper bound in \cref{thm:main} through the center-exclusion formulation. The diagonal model identifies the sufficient spectral constraints, and the trace comparison establishes them in general. The next section proves optimality of the constant.

\section{Sharpness via circular cones}\label{sec:sharpness}

The sharpness construction combines a factor $1/2$ in the semiaxis parallel to the cutting normal with a transverse volume factor tending to $\sqrt e$. Both features are realized by an explicit family of circular cones.

For $n\ge2$, write $x=(t,y)\in\R\times\R^{n-1}$ and set
\begin{equation}\label{eq:cone}
 K_n=\{(t,y):t\ge-1,\quad t+\sqrt{n^2-1}\,\|y\|\le n\}.
\end{equation}
Its vertex is $(n,0)$. Cut at $t=0$ and retain $K_n^-=K_n\cap\{t\le0\}$. Consider the ellipsoid
\begin{equation}\label{eq:cone-ellipsoid}
 E_n=-\frac12e_1+
 \diag\!\left(\frac12,\sqrt{\frac n{n-1}},\ldots,
 \sqrt{\frac n{n-1}}\right)\B.
\end{equation}
The left panel of \cref{fig:geometry} is this example for $n=2$.

\begin{lemma}[Normalization and feasibility for the cone family]\label{lem:cone-feasibility}
For the body and ellipsoid just defined, $J(K_n)=\B$ and $E_n\subset K_n^-$.
\end{lemma}

The proof in \cref{sec:cone-proof} gives explicit contact weights for the first assertion and checks the supporting halfspaces for the second. No maximality assertion about $E_n$ inside the half-cone is needed.

The semiaxes give the explicit volume-ratio lower bound
\begin{equation}\label{eq:cone-ratio}
 \frac{w(K_n^-)}{w(K_n)}
 \ge\frac{\vol(E_n)}{\vol(\B)}
 =\frac12\left(\frac n{n-1}\right)^{(n-1)/2}
 \longrightarrow\frac{\sqrt e}{2}.
\end{equation}
Indeed, with $m=n-1$, the logarithm of the transverse product is
\[
 \frac m2\log\left(1+\frac1m\right)\longrightarrow\frac12.
\]
Each transverse semiaxis tends to one, whereas their product tends to $\sqrt e$. This accumulated transverse contribution compensates for part of the contraction in the cutting direction.

\begin{proof}[Completion of the proof of \cref{thm:main}]
\Cref{sec:upper}, using the lemmas proved in \cref{sec:estimates}, establishes the upper bound for every convex body and central cut. For any smaller constant $r<\sqrt e/2$, the lower bound in~\eqref{eq:cone-ratio} exceeds $r$ in sufficiently large dimension. Hence no smaller constant is valid simultaneously in all dimensions.
\end{proof}

The cone example and its limiting value already appear in \citet{TarasovKhachiyanErlikh1988,Khachiyan1990}. The role of the construction here is to match the upper bound with an explicit normalized witness. Uniform sharpness is a statement over all dimensions; it does not imply equality in any fixed dimension. In dimension one, a body is an interval and the central cut halves its length, so the optimal ratio is $1/2$.

\section{Consequences of the spectral bounds}\label{sec:consequences}

The spectral constraints also yield a finite-dimensional estimate, strictness of the universal bound in finite dimension, and a necessary condition for near equality. We establish these consequences before presenting the supporting proofs. Throughout this section, $\E(a,A)$ satisfies the normalized hypotheses in~\eqref{eq:excluded-center} and $\alpha=\lambda_{\min}(A)$.

\subsection{Dimension-dependent estimates}
The estimate based on $\log z\le z-1$ eliminates the dimension dependence. Applying the arithmetic--geometric mean inequality directly to the same spectral constraints instead yields an explicit dependence on $n$. For $n\ge2$,
\begin{equation}\label{eq:finite-dimensional}
 \det A\le\alpha\left(\frac{n-f_j(\alpha)}{n-1}\right)^{2^j(n-1)},
 \qquad j=1,2.
\end{equation}
The bases are positive for a feasible nonsingular ellipsoid. These are useful bounds for a prescribed dimension, but we do not claim that optimizing them determines that dimension's best possible constant.

\subsection{Nonattainment and near equality}
For these conclusions it is useful to record the quantitative certificate proved in \cref{sec:scalar-proof}. Define
\[
 x=\log\alpha,\qquad x_0=-\log2,\qquad C=\frac12-\log2,
 \qquad F_j(x)=\log R_j(e^x).
\]
The global tangent bounds are
\begin{equation}\label{eq:tangents}
 F_1(x)\le C+\frac{x-x_0}{6},\qquad
 F_2(x)\le C-\frac{x-x_0}{42}.
\end{equation}
For $x<x_0$, the first upper bound is strictly below $C$; for $x>x_0$, the second is strictly below $C$. These inequalities establish nonattainment away from $\alpha=1/2$ and quantitatively constrain the smallest semiaxis of a near-extremizer.

\begin{corollary}\label{cor:nonattainment}
In finite dimension, every nonsingular ellipsoid satisfying the assumptions in~\eqref{eq:excluded-center} has $\det A<\sqrt e/2$.
\end{corollary}
\begin{proof}
If $\alpha\ne1/2$, the relevant tangent bound in~\eqref{eq:tangents} is strictly below $C$. If $\alpha=1/2$ and $n\ge2$, use $j=1$ in the finite-dimensional bound~\eqref{eq:finite-dimensional}:
\[
 \det A\le\frac12\left(1+\frac1{4(n-1)}\right)^{2(n-1)}
 <\frac12e^{1/2}.
\]
For $n=1$ and $\alpha=1/2$, the determinant is $1/2$.
\end{proof}

\begin{corollary}\label{cor:near-equality}
Under the same assumptions,
\begin{equation}\label{eq:spectral-gap}
 \log\frac{\rstar}{\det A}\ge
 \begin{cases}
  \dfrac16\log\dfrac1{2\alpha},&0<\alpha\le\tfrac12,\\[5pt]
  \dfrac1{42}\log(2\alpha),&\alpha\ge\tfrac12.
 \end{cases}
\end{equation}
In particular, if $\det A\ge\rstar e^{-\varepsilon}$ for $\varepsilon\ge0$, then
\begin{equation}\label{eq:near-equality}
 \frac12e^{-6\varepsilon}\le\alpha\le\frac12e^{42\varepsilon}.
\end{equation}
\end{corollary}
\begin{proof}
Combine $\log\det A\le F_j(x)$ with the first tangent bound when $x\le x_0$ and the second when $x\ge x_0$. Subtracting from $C$ gives inequality~\eqref{eq:spectral-gap}; solving for $\alpha$ gives the bounds~\eqref{eq:near-equality}.
\end{proof}

This necessary condition controls only the smallest semiaxis. A geometric stability theorem would additionally require control of the remaining semiaxes, the center displacement, and the surrounding body.

\subsection{Repeated exact central cuts}
For an idealized sequence
\[
 K_{k+1}=K_k\cap H_k,\qquad c(K_k)\in\partial H_k,
\]
successive applications of the main theorem give
\begin{equation}\label{eq:iteration}
 w(K_N)\le\rstar^Nw(K_0),\qquad
 \log\frac{w(K_0)}{w(K_N)}\ge N\left(\log2-\frac12\right).
\end{equation}
This is the optimal universal one-step factor for the exact-center volume measure. An algorithmic complexity claim would also have to account for finding the centers, numerical accuracy, and separation costs. The geometric inequality alone does not provide those additional estimates.

\section{Proofs of the supporting results}\label{sec:estimates}

This section proves the results used in the preceding argument: the geometric reduction, the directional rank-one estimate, the joint scalar bound, and the cone certificate. The proofs are independent of the main theorem, so the deferred presentation introduces no circular dependence.

\subsection{Center exclusion and first-order maximality}\label{sec:basic-proofs}

\begin{proof}[Proof of \cref{prop:exclusion}]
The equation $0=a+Ay$ has the unique solution $y=-A^{-1}a$. Thus $0\notin\interior\E(a,A)$ is equivalent to $\|A^{-1}a\|\ge1$. Any ellipsoid in a halfspace through $0$ has this property, so the implication~\eqref{eq:excluded-center} suffices for the central-cut theorem.

Conversely, suppose $\rho=\|A^{-1}a\|\ge1$ and set $v=A^{-2}a\ne0$. The minimum of the linear functional $v^\top x$ on the ellipsoid is
\[
 \min_{x\in\E(a,A)}v^\top x
 =v^\top a-\|Av\|=\rho^2-\rho\ge0.
\]
Hence the ellipsoid lies in the closed halfspace $v^\top x\ge0$. A central-cut upper bound applied to that halfspace implies~\eqref{eq:excluded-center}, proving equivalence.
\end{proof}

\begin{proof}[Proof of \cref{lem:trace}]
Convexity provides a feasible path from the unit ball toward $\E(b,S)$:
\[
 tb+((1-t)I+tS)\B\subset K,\qquad0\le t\le1.
\]
Indeed, the point associated with $y\in\B$ is the convex combination $(1-t)y+t(b+Sy)$. Its shape matrix stays positive definite. Since the unit ball has maximum volume,
\[
 \log\det((1-t)I+tS)\le0.
\]
The right derivative at $t=0$ is $\tr(S-I)$ and cannot be positive. Therefore $\tr S\le n$.
\end{proof}

\subsection{Containment of the intermediate ellipsoid}\label{sec:intermediate-proof}
Containment can be checked one supporting direction at a time. For an ellipsoid, the support function is
\begin{equation}\label{eq:support}
 h_{\E(b,S)}(v)=v^\top b+\|Sv\|,
\end{equation}
obtained by maximizing $(Sv)^\top y$ over $y\in\B$. This identity provides an algebraic criterion for the required containment.

\begin{proof}[Proof of \cref{lem:intermediate}]
Fix $v\ne0$ and put $h=h_K(v)$ and $s=v^\top b$. The assumed containments give
\[
 \|v\|\le h,\qquad \|Sv\|\le h-s.
\]
In particular $h>0$ and $s\le h$. Cauchy--Schwarz therefore implies
\[
 v^\top Sv\le\|v\|\,\|Sv\|\le h(h-s).
\]
Adding the rank-one term completes a square:
\[
 v^\top(S+bb^\top/4)v
 \le h^2-hs+s^2/4=(h-s/2)^2.
\]
Since $h-s/2>0$, the support function of the proposed ellipsoid is at most
\[
 \frac s2+\sqrt{v^\top(S+bb^\top/4)v}\le h.
\]
It satisfies every support-function bound of $K$, which proves the containment in~\eqref{eq:intermediate}.
\end{proof}

\subsection{The directional rank-one estimate}\label{sec:rank-one-proof}
The matrix facts needed here are proved in \cref{app:matrix}: inverse and negative fractional powers reverse the positive definite order, and differentiation after taking the trace gives $D[\tr(X^p)](H)=p\tr(X^{p-1}H)$. These facts hold without commuting matrices.

\begin{proof}[Proof of \cref{lem:directional}]
Since $\|u\|=1$, we have $uu^\top\preceq I$ and therefore $bb^\top\preceq A^2$. Combining this with the assumed baseline comparison gives, for $s\ge0$,
\begin{equation}\label{eq:path-comparison}
 M+sbb^\top\preceq\frac{m+s\alpha^2}{\alpha^2}A^2.
\end{equation}
The exponent $p-1$ belongs to $(-1,0)$. Its operator antitonicity reverses the comparison:
\[
 (M+sbb^\top)^{p-1}
 \succeq\left(\frac{m+s\alpha^2}{\alpha^2}\right)^{p-1}A^{2p-2}.
\]
Evaluate both sides on $b=Au$. Since $u^\top A^{2p}u\ge\alpha^{2p}$,
\begin{equation}\label{eq:directional-derivative-bound}
 b^\top(M+sbb^\top)^{p-1}b
 \ge\alpha^2(m+s\alpha^2)^{p-1}.
\end{equation}
On the other hand, the trace derivative is
\[
 \frac{d}{ds}\tr\bigl((M+sbb^\top)^p\bigr)
 =p\,b^\top(M+sbb^\top)^{p-1}b.
\]
Multiply inequality~\eqref{eq:directional-derivative-bound} by $p$ and integrate from $0$ to $t$. The resulting scalar integral is $(m+t\alpha^2)^p-m^p$, proving inequality~\eqref{eq:directional}.
\end{proof}

The matrix $A$ both bounds the baseline from above and parametrizes the displacement through $b=Au$. The eigenvalue estimate concerns only a power of $A$; simultaneous diagonalization of $M$ and $bb^\top$ is not required. Likewise, the derivative formula is for the trace; it does not assert that the full matrix derivative of $X^p$ is multiplication by $pX^{p-1}$.

\begin{remark}[Exactness in the shortest-axis case]\label{rem:equality-model}
If $u$ is a unit eigenvector of $A$ for $\alpha$ and $M=(m/\alpha^2)A^2$, the perturbation changes only the eigenvalue $m$ in that direction. Equality then holds in~\eqref{eq:directional}. This is the equality case underlying the scalar functions $f_1$ and $f_2$.
\end{remark}

\subsection{The joint scalar bound}\label{sec:scalar-proof}

\begin{proof}[Proof of \cref{lem:scalar} and the tangent bounds~\eqref{eq:tangents}]
Put
\[
 F_j(x)=\log R_j(e^x)=x+2^j(1-f_j(e^x)).
\]
The required concavity holds in this logarithmic coordinate. To verify it, write
\begin{align*}
 \log f_1(e^x)&=\frac12\log(e^x+e^{2x}/4),\\
 \log f_2(e^x)&=\frac12\log\bigl(e^{\log f_1(e^x)}+e^{2x}/16\bigr).
\end{align*}
The function $(s,t)\mapsto\log(e^s+e^t)$ is convex and nondecreasing in both variables. The first displayed expression is therefore convex, and composition gives convexity of the second. Hence $f_j(e^x)$ are positive log-convex functions, in particular convex functions. It follows that $F_1,F_2$ are concave.

At $x_0=-\log2$, direct calculation gives
\begin{gather}
 f_1(1/2)=\frac34,\qquad f_2(1/2)=\frac78,\label{eq:crossing-values}\\
 f_1'(1/2)=\frac56,\qquad f_2'(1/2)=\frac{43}{84}.\label{eq:crossing-derivatives}
\end{gather}
With $C=1/2-\log2$, this means
\begin{equation}\label{eq:F-data}
 F_1(x_0)=F_2(x_0)=C,\qquad
 F_1'(x_0)=\frac16,\quad F_2'(x_0)=-\frac1{42}.
\end{equation}
Each concave function lies below its tangent at $x_0$, proving both bounds in~\eqref{eq:tangents}. If $x\le x_0$, the first tangent is at most $C$; if $x\ge x_0$, the second is at most $C$. Thus
\[
 \min\{F_1(x),F_2(x)\}\le C\qquad(x\in\R).
\]
At $x=x_0$ both functions equal $C$. Exponentiating proves the exact maximum in~\eqref{eq:scalar-sharp}.
\end{proof}

The two parameter ranges in \cref{fig:bounds} are therefore controlled by different constraints. The tangent inequalities determine and certify the joint maximum analytically.

\begin{remark}[An equivalent weighted certificate]\label{rem:weighted}
An equivalent certificate combines the two tangent inequalities. The weights $1/8$ and $7/8$ cancel the slopes in~\eqref{eq:F-data}, so the concave function
\[
 G(x)=\frac18F_1(x)+\frac78F_2(x)
\]
has derivative zero at $x_0$ and global maximum $C$. Since $\min(F_1,F_2)\le G$, this also proves the bound. The weights follow from $\theta/6-(1-\theta)/42=0$, rather than numerical fitting.
\end{remark}

For comparison with the 1990 estimate, substituting $\alpha=d^{-2}$ into $R_1$ gives
\[
 d^{-2}\exp\!\left(2-\frac{2\sqrt{1+1/(4d^2)}}d\right).
\]
This is Khachiyan's final expression. Its maximizer satisfies $4d^6-3d^4-4d^2-1=0$. The additional bound $R_2$ is restrictive near that maximizer, while $R_1$ is restrictive near the maximizer of $R_2$. Their complementarity is the improvement used here.

\subsection{The cone certificate}\label{sec:cone-proof}

\begin{proof}[Proof of \cref{lem:cone-feasibility}]
Every point $(t,y)\in\B$ satisfies $t\ge-1$, and Cauchy--Schwarz gives
\[
 t+\sqrt{n^2-1}\,\|y\|\le n\sqrt{t^2+\|y\|^2}\le n
\]
on the ball. Hence $\B\subset K_n$. The base contact is $u_0=-e_1$, and the side contacts are
\[
 u_v=\left(\frac1n,\sqrt{1-\frac1{n^2}}\,v\right),\qquad\|v\|=1.
\]
Choose $v$ among the signed coordinate vectors of $\R^{n-1}$. Give $u_0$ weight $c_0=n/(n+1)$ and each of the $2(n-1)$ chosen side contacts weight
\[
 c_v=\frac{n^2}{2(n^2-1)}.
\]
The total side weight is $n^2/(n+1)$, so the first-coordinate mean is $-n/(n+1)+n/(n+1)=0$. Transverse means and mixed second moments vanish in signed pairs. The first-coordinate second moment is $n/(n+1)+1/(n+1)=1$, and each transverse second moment is $2c_v(1-1/n^2)=1$. Thus
\begin{equation}\label{eq:cone-certificate}
 \sum_i c_i u_i=0,\qquad\sum_i c_i u_i u_i^\top=I.
\end{equation}
These are supporting directions for $K_n$. The sufficient contact certificate in \cref{lem:contact-certificate} therefore gives $J(K_n)=\B$.

For the second assertion, the first coordinate of $E_n$ lies in $[-1,0]$. At every side normal $u_v$, its support function is
\[
 -\frac1{2n}+\sqrt{\frac1{4n^2}+\frac n{n-1}\left(1-\frac1{n^2}\right)}
 =-\frac1{2n}+1+\frac1{2n}=1.
\]
It therefore satisfies the base, cut, and all side inequalities, proving $E_n\subset K_n^-$.
\end{proof}

This completes the proofs of the supporting results used in \cref{sec:architecture,sec:upper,sec:sharpness}. The standard matrix facts and the contact-certificate argument are supplied in the appendices.

\section{Related work and further directions}\label{sec:history}

The diagonal model identifies the spectral constraints required for the determinant estimate, and the trace comparison extends them to arbitrary displacements. The corresponding geometric construction, center-computation methods, and matrix inequalities belong to distinct strands of the literature.

\subsection{Geometric foundations}
Extremal ellipsoids were studied well before their use in cutting-plane methods. Behrend treated the smallest circumscribed and largest inscribed ellipses in the plane \citep{Behrend1938}. John's variational approach \citep{John1948} and subsequent work, including that of Danzer, Laugwitz, and Lenz \citep{DanzerLaugwitzLenz1957}, underlie the modern theory. \citet{GulerGurtuna2007} give an account of existence, uniqueness, contact conditions, symmetry, and examples, together with a discussion of the terminology.

The argument uses three ingredients from this theory. The maximum-volume ellipsoid provides a normalization, its maximality provides a trace constraint, and a sufficient contact certificate verifies the cone. The necessity part of the full John contact-point characterization is not used. This permits the elementary, self-contained treatment in \cref{app:ellipsoids}.

\paragraph{Ordinary volume and centroid cuts.}
For ordinary volume, \citet[Theorem~2 and its proof]{Grunbaum1960} establishes a sharp inequality for cuts through the centroid $g(K)$. Applying it to the two closed halfspaces determined by a hyperplane through $g(K)$ gives
\[
 \gamma_n\le\frac{\vol(K\cap H)}{\vol(K)}\le1-\gamma_n,
 \qquad \gamma_n=\left(\frac n{n+1}\right)^n.
\]
The upper bound follows from the lower bound for the opposite halfspace, and both dimension-dependent constants are sharp. Since $\gamma_n>e^{-1}$ and $\gamma_n\to e^{-1}$, the optimal dimension-independent upper factor for these ordinary-volume ratios is $1-e^{-1}$. This result controls $\vol(K\cap H)/\vol(K)$ at the centroid, whereas \cref{thm:main} controls $w(K\cap H)/w(K)$ at the center $c(K)$ of the maximum-volume inscribed ellipsoid. The two centers need not coincide; the constants therefore concern different volume functionals and cutting rules.

\subsection{Cutting-plane geometry and center computation}
The method of inscribed ellipsoids of \citet{TarasovKhachiyanErlikh1988} uses exact or sufficiently accurate approximate centers to localize solutions of convex optimization problems. Their Russian original states the central-cut factor $0.843$, displays a recursive estimate $0.8429\ldots$, and gives the cone limit $\sqrt e/2$. It also proves the weaker $8/9$ bound as an intermediate step. Khachiyan's 1990 paper gives its own $0.844\ldots$ proof and explicitly proposes the sharp-constant conjecture \citep{Khachiyan1990}. The intermediate construction and its recursive use therefore have established precedents. The result here comes from the quantitative trace comparison and the simultaneous use of the two resulting bounds.

The contraction factor depends on the position of the cutting hyperplane. \citet{PrimakKheyfets1995} obtain a factor $1/2$ in a modified method using a cut tangent to the current maximum-volume ellipsoid. Its hypothesis differs from the central-cut condition of \cref{thm:main}. That smaller factor is therefore not a bound for the problem considered here.

Computing the center is a separate optimization task. \citet{KhachiyanTodd1993} analyze approximation through a sequence of maximal-inscribed-paraboloid problems. \citet{Anstreicher2002} improves the complexity of inscribed-ellipsoid approximation, and \citet{ZhangGao2003} study numerical formulations and practical primal-dual methods. More recently, \citet{Sun2026} analyzes a polarity iteration using minimum-volume covering ellipsoids of shifted polars, with instance-dependent convergence. These works address the cost or convergence of finding an ellipsoid; the present theorem addresses the geometric effect of a cut after an exact center is known.

For approximate centers, \citet[Lemma~1]{WadaFujisaki2010} use a factor $8/(9\eta^2)$ when the available ellipsoid has volume at least an $\eta$ fraction of the optimum. Our trace constraint uses exact maximality. Replacing that approximate-center factor by an expression involving $\sqrt e/2$ would require a separate argument. Broader accounts of Khachiyan's geometric and algorithmic contributions are given by \citet{BorosGurvich2008} and \citet{ElbassioniDumitrescu2017}.

\subsection{Trace inequalities and the directional estimate}
For suitable trace functions, the dependence of an increment on the matrix baseline admits a noncommutative monotonicity principle. \citet{BouhtouGaubertSagnol2010} use strong subadditivity of fractional trace powers in experimental design. \citet[Proposition~2.3]{Sagnol2013} proves, for a continuous $f$ on $[0,\infty)$ with operator-antitone derivative on $(0,\infty)$,
\begin{equation}\label{eq:known-SSA}
 \tr f(X+Y+Z)+\tr f(Z)
 \le\tr f(X+Z)+\tr f(Y+Z)
\end{equation}
for positive semidefinite arguments. For $f(x)=x^p$, $0<p<1$, the derivative $px^{p-1}$ is operator antitone. Recent discussion and further examples appear in \citet{Niculescu2025}; the fractional-power statement is proved here in \cref{app:matrix}.

\Cref{lem:directional} uses this antitonicity and trace-calculus mechanism in a quantitative directional setting. The baseline comparison with $A^2$ and the correlated vector $b=Au$ preserve the displacement information needed by the geometric application. The estimate is stated independently of ellipsoids, so it may be studied as a matrix tool in other problems. Its application here is precise: one square-root estimate and two telescoping estimates yield the pair of constraints that the diagonal model suggested.

\subsection{Further directions}\label{sec:discussion}
The dimension-independent question is settled by the matching upper bound and cone limit. The best constant for each prescribed dimension remains undetermined by this argument. The near-equality condition in \cref{cor:near-equality} controls only the smallest semiaxis; a geometric stability result would also need information about the other axes, the displacement, and the body. A robust analogue for approximate centers would connect the exact theorem more directly to implementations of cutting-plane methods.

The argument relies on the interaction of two spectral constraints. The diagonal model identifies their form, matrix analysis establishes them for arbitrary displacements, and convex analysis determines the parameter range controlled by each bound. Their combination yields the dimension-independent constant.

\paragraph{Proof provenance and use of AI.}
The proof was discovered by an AI language model in response to a research problem posed by the first author. The development included the matrix estimates, literature comparisons, exact scalar calculations, and numerical error checks. AI assistance was also used to develop the Lean~4 formalization. The resulting formal proofs have been checked by Lean's kernel, with the scope and audit described in \cref{app:formal-verification}. The human authors are responsible for the mathematical interpretation, literature attribution, and final submitted manuscript.

\label{page:main-end}

\clearpage
\appendix
\section{Basic facts on maximal inscribed ellipsoids}\label{app:ellipsoids}

This appendix supplies the existence, uniqueness, and certificate facts used in the main text. It does not require the necessity part of the general contact-point characterization.

\begin{proposition}\label{prop:exist-unique}
Every convex body $K\subset\R^n$ contains a unique maximum-volume ellipsoid.
\end{proposition}
\begin{proof}
Consider pairs $(a,A)$ with $A\succeq0$ and $a+A\B\subset K$. This set is closed. Moreover, $a\in K$, and for every unit vector $u$, the two points $a\pm Au$ lie in $K$, so $2\|Au\|\le\operatorname{diam}K$. Thus the feasible pairs form a compact set. The determinant attains a maximum, which is positive because $K$ contains a ball of positive radius. A maximizing matrix is therefore positive definite.

The feasible set of pairs is convex, as follows by taking pointwise convex combinations with the same $y\in\B$. Suppose $(a_0,A_0)$ and $(a_1,A_1)$ both maximize the determinant. Strict concavity of $\log\det$ on the positive definite cone implies $A_0=A_1$. For completeness, along $S(t)=A_0+tD$, with $D=A_1-A_0\ne0$,
\[
 \frac{d^2}{dt^2}\log\det S(t)
 =-\tr\bigl(S(t)^{-1}D S(t)^{-1}D\bigr)<0,
\]
since the trace on the right is the squared Frobenius norm of $S(t)^{-1/2}D S(t)^{-1/2}$.

If the centers were distinct, an affine normalization would give two maximizing ellipsoids $\B$ and $b+\B$, with $b\ne0$, inside the transformed body. \Cref{lem:intermediate}, which does not assume maximality, produces an ellipsoid with shape $(I+bb^\top/4)^{1/2}$. Its determinant is $\sqrt{1+\|b\|^2/4}>1$, a contradiction. Thus the centers agree as well.
\end{proof}

\begin{lemma}[Sufficient contact certificate]\label{lem:contact-certificate}
Suppose $\B\subset K$ and there are unit vectors $u_i$ and positive weights $c_i$ such that
\[
 K\subset\{x:u_i^\top x\le1\}\quad\text{for every }i,
 \qquad \sum_i c_i u_i=0,\quad \sum_i c_i u_i u_i^\top=I.
\]
Then $J(K)=\B$.
\end{lemma}
\begin{proof}
For every $\E(a,A)\subset K$,
\[
 u_i^\top a+\|Au_i\|\le1,
 \qquad u_i^\top Au_i\le\|Au_i\|.
\]
Multiply by $c_i$ and sum. The assumed identities give $\tr A\le\sum_i c_i=n$. The arithmetic--geometric mean inequality yields $\det A\le(\tr A/n)^n\le1$. Since the ball is feasible, it is maximizing, and uniqueness follows from \cref{prop:exist-unique}.
\end{proof}

This is the sufficient direction of the standard contact certificate used in extremal-ellipsoid theory; see \citet{GulerGurtuna2007}. In \cref{sec:cone-proof}, all supporting directions and weights are given explicitly.

\section{Fractional trace powers and diminishing increments}\label{app:matrix}

We give the matrix facts used in \cref{sec:estimates}. They are classical; integral proofs are included to specify the order and differentiability assertions without a commutativity assumption.

\begin{lemma}\label{lem:power-order}
If $0\prec X\preceq Y$, then $Y^{-1}\preceq X^{-1}$. For $0<\theta<1$,
\[
 X^\theta\preceq Y^\theta,
 \qquad X^{-\theta}\succeq Y^{-\theta}.
\]
\end{lemma}
\begin{proof}
Write $Y=X^{1/2}CX^{1/2}$ with $C\succeq I$. Then $C^{-1}\preceq I$ and $Y^{-1}=X^{-1/2}C^{-1}X^{-1/2}\preceq X^{-1}$.

The scalar beta integral, transferred to positive definite matrices by orthogonal diagonalization, gives
\begin{align}
 X^\theta&=\frac{\sin(\pi\theta)}\pi
 \int_0^\infty s^{\theta-1}X(X+sI)^{-1}\,ds,\label{eq:positive-power-integral}\\
 X^{-\theta}&=\frac{\sin(\pi\theta)}\pi
 \int_0^\infty s^{-\theta}(X+sI)^{-1}\,ds.\label{eq:negative-power-integral}
\end{align}
These integrals converge in norm in finite dimension. The negative-power comparison follows immediately from the inverse comparison. For positive powers, use $X(X+sI)^{-1}=I-s(X+sI)^{-1}$ and integrate the resulting order inequality with positive weights.
\end{proof}

\begin{lemma}[Trace differentiation]\label{lem:trace-derivative}
Let $0<p<1$, $X\succ0$, and let $H$ be symmetric. Then
\begin{equation}\label{eq:trace-derivative}
 D\bigl[\tr(X^p)\bigr](H)=p\,\tr(X^{p-1}H).
\end{equation}
\end{lemma}
\begin{proof}
Take the trace in the integral representation~\eqref{eq:positive-power-integral} and differentiate in the direction $H$. With $R_s=(X+sI)^{-1}$, differentiation of $I-sR_s$ and cyclicity of the trace give
\[
 D\bigl[\tr(X^p)\bigr](H)
 =\frac{\sin(\pi p)}\pi\int_0^\infty s^p\tr(R_s^2H)\,ds.
\]
Differentiation under the integral is valid locally uniformly on the positive definite cone: the integrand is bounded by an integrable multiple of $s^p$ near zero and $s^{p-2}$ at infinity. In an eigenbasis of $X$, the scalar identity
\[
 \frac{\sin(\pi p)}\pi\int_0^\infty
 \frac{s^p}{(\lambda+s)^2}\,ds=p\lambda^{p-1}
\]
proves the trace identity~\eqref{eq:trace-derivative}. The scalar identity follows by integration by parts from the negative-power integral. Repeated eigenvalues cause no difficulty.
\end{proof}

In particular, setting $H=bb^\top$ proves the derivative formula used in \cref{lem:directional}. This trace identity does not imply that the full derivative of $X^p$ is multiplication by $pX^{p-1}$.

\begin{proposition}[Diminishing trace increments]\label{prop:diminishing}
Let $0<p<1$, $0\prec X\preceq Y$, and $H\succeq0$. Then
\begin{equation}\label{eq:diminishing}
 \tr((X+H)^p)-\tr(X^p)
 \ge\tr((Y+H)^p)-\tr(Y^p).
\end{equation}
Equivalently, for $U,V,W\succeq0$,
\begin{equation}\label{eq:SSA-powers}
 \tr((U+V+W)^p)+\tr(W^p)
 \le\tr((U+W)^p)+\tr((V+W)^p).
\end{equation}
\end{proposition}
\begin{proof}
Put $Z=Y-X\succeq0$ and
\[
 g(t)=\tr((X+tZ+H)^p)-\tr((X+tZ)^p),\qquad 0\le t\le1.
\]
By \cref{lem:trace-derivative,lem:power-order},
\[
 g'(t)=p\,\tr\left(\bigl[(X+tZ+H)^{p-1}-(X+tZ)^{p-1}\bigr]Z\right)\le0.
\]
Indeed, the bracket is negative semidefinite, so its product with $Z$ has nonpositive trace, as is seen by conjugating with $Z^{1/2}$. Hence $g(0)\ge g(1)$, proving inequality~\eqref{eq:diminishing}. Taking $X=W$, $Y=V+W$, $H=U$ proves inequality~\eqref{eq:SSA-powers} for $W\succ0$. Replacing $W$ by $W+\varepsilon I$ and letting $\varepsilon\downarrow0$ extends it to the stated semidefinite case. The reverse implication follows from the same substitutions.
\end{proof}

The proposition is the fractional-power case of the established strong-subadditivity principle in \citet{BouhtouGaubertSagnol2010,Sagnol2013}; see also \citet[Corollary~4]{Niculescu2025}. Its proof is included for self-containment, not as a separate novelty claim. The explicit lower bound in \cref{lem:directional} additionally uses the comparison with $A^2$ and the correlated vector $b=Au$.

\section{A resolvent proof of the directional estimate}\label{app:resolvent}

An alternative proof of \cref{lem:directional} separates the quantitative directional comparison from a general rank-one integral identity.

Under the assumptions of that lemma, adding $sI$ to the matrix comparison~\eqref{eq:M-comparison} and then inverting gives, for $s>0$,
\begin{align}
 b^\top(M+sI)^{-1}b
 &\ge u^\top A^2\left(\frac m{\alpha^2}A^2+sI\right)^{-1}u\notag\\
 &\ge\frac{\alpha^2}{m+s}.\label{eq:resolvent-bound}
\end{align}
For the second inequality, diagonalize $A$ and use monotonicity of
\[
 \lambda\longmapsto\frac{\lambda^2}{(m/\alpha^2)\lambda^2+s}
\]
on $(0,\infty)$. Every eigenvalue of $A$ is at least $\alpha$, and $\|u\|=1$. The first comparison makes no assumption that $M$ commutes with $A$.

Let $N=M+tbb^\top$. For $0<p<1$, the rank-one identity is
\begin{equation}\label{eq:rank-one-integral}
 \tr(N^p)-\tr(M^p)
 =\frac{p\sin(\pi p)}\pi\int_0^\infty
 s^{p-1}\log\!\left(1+t\,b^\top(M+sI)^{-1}b\right)\,ds.
\end{equation}
To establish this identity, take the trace difference in the integral representation~\eqref{eq:positive-power-integral}. Writing
\[
 L(s)=\log\det(N+sI)-\log\det(M+sI),
\]
the result is
\[
 \tr(N^p)-\tr(M^p)
 =-\frac{\sin(\pi p)}\pi\int_0^\infty s^p L'(s)\,ds.
\]
Integration by parts replaces the right-hand side by $p\sin(\pi p)/\pi$ times $\int_0^\infty s^{p-1}L(s)\,ds$. The determinant lemma identifies
\[
 L(s)=\log\!\left(1+t\,b^\top(M+sI)^{-1}b\right),
\]
which proves identity~\eqref{eq:rank-one-integral}. All boundary terms vanish: $L(s)$ is bounded as $s\downarrow0$, while $L(s)=O(s^{-1})$ as $s\to\infty$, and $0<p<1$.

Finally insert the resolvent bound~\eqref{eq:resolvent-bound} into the integral identity~\eqref{eq:rank-one-integral}. The resulting lower bound is
\[
 \frac{p\sin(\pi p)}\pi\int_0^\infty
 s^{p-1}\log\!\left(1+\frac{t\alpha^2}{m+s}\right)\,ds
 =(m+t\alpha^2)^p-m^p,
\]
where the equality is the scalar case of the same integral identity. This proves inequality~\eqref{eq:directional} without differentiating along the perturbation path.

\clearpage
\section{Formal verification}\label{app:formal-verification}

The main theorem and the supporting mathematical results have been formalized in Lean~4 \citep{deMouraUllrich2021} using mathlib \citep{Mathlib2020}. The development covers actual ellipsoids in Euclidean space and Lebesgue volume, including existence and uniqueness of the maximal ellipsoid, affine normalization, the central-cut bound in every positive dimension, and the cone construction proving uniform optimality. It also verifies the finite-dimensional estimates, nonattainment, the exact one-dimensional ratio, the near-equality constraints, and the iteration bounds.

\paragraph{Statements and representations.}
The geometric space is \texttt{EuclideanSpace} over $\R$ with coordinates indexed by \texttt{Fin n}. Shape matrices are real and positive definite, matrix comparisons use the Loewner order, and fractional powers use continuous functional calculus. The volume functional is the supremum of the Lebesgue volumes of feasible ellipsoids; positivity, finiteness, and attainment are proved. Thus the final geometric theorem has the hypotheses of \cref{thm:main}, without additional trace constraints or commutativity assumptions. Selected declaration names are listed below; each belongs to the namespace \texttt{Khachiyan}.

\begin{center}
\small
\begin{tabular}{@{}p{0.43\textwidth}p{0.53\textwidth}@{}}
\toprule
Result & Lean declaration \\
\midrule
Geometric upper bound & \path{t01} \\
Uniform optimality & \path{sharp_constant} \\
Existence and uniqueness & \path{existsUnique_isMaxDet} \\
Directional estimate: trace proof & \path{RankOne.trace_increment_lower_bound} \\
Directional estimate: resolvent proof & \path{Resolvent.trace_increment_lower_bound} \\
Finite-dimensional nonattainment & \path{normalized_det_lt_rStar}, \path{t03} \\
Near equality and repeated cuts & \path{t04}, \path{t05} \\
\bottomrule
\end{tabular}
\end{center}

\paragraph{Independent routes and verification scope.}
Both proofs of \cref{lem:directional} are formalized for every $0<p<1$, every $t\ge0$, and arbitrary unit directions. The resolvent route shares spectral and order foundations with the trace route. Its actual import closure and transitive proof dependencies exclude the trace-derivative and primary rank-one modules, and its final theorem has the same hypotheses and conclusion. The formalization uses equivalent implementations of some intermediate calculations: a spectral midpoint argument proves uniqueness, and the integral formulas use the positive normalization $(\int_0^\infty s^{p-1}/(1+s)\,ds)^{-1}$. The explicit evaluation of this normalization as $\sin(\pi p)/\pi$ is not a separate verified theorem. The verification claim concerns the mathematical statements and their proofs, rather than a line-by-line transcription of every displayed calculation or the historical survey.

\paragraph{Reproducibility and trusted foundations.}
The development fixes Lean \texttt{v4.35.0-rc2} and mathlib commit \path{0a6c8e0355da0405d616f80b9f8232c4fab2cc5b}. A fresh verification on September~23, 2026 rebuilt the project and its pinned library dependencies from source in a clean checkout, using the installed Lean toolchain. The complete audit checked 598 declarations; a separate audit checked all 43 public declarations of the resolvent route. Their axiom dependencies contain only Lean's standard foundational axioms \texttt{propext}, \texttt{Classical.choice}, and \texttt{Quot.sound}. No proof placeholder, project-specific axiom, or external numerical oracle occurs in the certified chain.

\paragraph{Source availability.}
The formal-verification source is publicly available at \url{https://github.com/DrZhouKarl/KhachiyanEllipsoidConjecture}. The repository includes the Lean sources, pinned version files, a statement map, and \path{scripts/verify.py}, which reruns the build, the axiom and independence audits, and the comparison of the two rank-one theorem statements. The published sources and version files match the audited development byte for byte.

\clearpage
\bibliographystyle{plainnat}
\bibliography{references}
\end{document}